\documentclass{article}
\usepackage{amsmath,amssymb,latexsym,amsthm,mathrsfs}

\newtheorem{theo}{Theorem}[section]

\newtheorem{cor}[theo]{Corollary}

\newcommand{\ra}{\rightarrow}
\theoremstyle{definition}
\newtheorem{defin}[theo]{Definition}

\newtheorem{exa}{Example}[section]

\theoremstyle{remark}
\newtheorem{rem}[theo]{Remark}

\begin{document}

\title{Random walks under
slowly varying moment conditions on groups of polynomial volume growth}
\author{Laurent Saloff-Coste\thanks{%
Both authors partially supported by NSF grants 
DMS 1004771 and DMS 1404435} \\
{\small Department of Mathematics}\\
{\small Cornell University} \and Tianyi Zheng \\
{\small Department of Mathematics}\\
{\small Stanford University} }
\maketitle

\begin{abstract}
Let $G$ be a finitely generated group of polynomial volume growth 
equipped with a word-length $|\cdot|$.
The goal of this paper is to develop techniques to study the behavior 
of random walks driven by symmetric measures $\mu$ such that, 
for any $\epsilon>0$, 
$\sum|\cdot|^\epsilon\mu=\infty$. In particular, we provide a sharp lower 
bound for the return probability  in the case when $\mu$ 
has a finite weak-logarithmic moment.
\end{abstract}

\section{Introduction}

\setcounter{equation}{0}

Let $G$ be a finitely generated group. Let $S=(s_1,\dots,s_k)$ be a generating 
$k$-tuple and $S^*=\{e,s_1^{\pm1},\cdots, s_k^{\pm 1}\}$ 
be the associated symmetric generating set. 
Let $|\cdot|$ be the associated word-length 
so that $|g|$ is the least integer 
$m$ such that $g=\sigma_1\dots\sigma_m$ with 
$\sigma_i \in \mathcal S^*$ (and the convention that $|e|=0$).

Given two monotone functions $f_1,f_2$, write $f_1\simeq f_2$ 
if there exists $c_i\in (0,\infty)$ such that 
$c_1f_1(c_2t)\le f_2(t)\le c_3f_1(c_4t)$ on the domain of definition of $f_1,f_2$
(usually, $f_1,f_2$ are defined on a neighborhood of $0$ or infinity and tend to $0$ or infinity at either $0$ or infinity. In some cases, 
one or both functions are defined only on a countable set such as $\mathbb N$). 

In \cite{PSCstab} it is proved that there exists a function 
$\Phi_G:\mathbb N\ra (0,\infty)$ such that, 
if $\mu$ is a symmetric probability measure with generating 
support and finite second moment, that is $\sum |g|^2\mu(g)<\infty$, then
$$\mu^{(2n)}(e)\simeq \Phi_G(n).$$

In \cite{BSClmrw}, A. Bendikov and the first author considered the question of
finding lower bounds for the probability of return $\mu^{(2n)}(e)$ when $\mu$
is only known to have a finite moment of some given exponent lower than $2$. 
Very generally, let $\rho: [0,\infty)\ra [1,\infty)$ be  given.

We say that a measure $\mu$ has finite $\rho$-moment if 
$\sum \rho(|g|)\mu(g)<\infty$. We say that $\mu$ has finite weak-$\rho$-moment
if 
\begin{equation} \label{weak}
W(\rho,\mu)
:= \sup_{s>0} \left\{s \mu(\{g: \rho(|g|) >s\})\right\} <\infty. \end{equation}

\begin{defin}[Fastest decay under $\rho$-moment] \label{def1}
Let $G$ be a countable group.
Fix a function $\rho:[0,\infty)\ra[1,\infty)$ with $\rho(0)= 1$.
Let $\mathcal S_{G,\rho}$ be the set of all symmetric
probability  $\phi$ on $G$ with the properties that
$\sum \rho(|g|)\phi(g) \le 2$. Set
$$\Phi_{G,\rho}: n\mapsto \Phi_{G,\rho}(n)= \inf\left\{
\phi^{(2n)}(e) : \phi \in \mathcal S_{G,\rho}\right\}. $$
\end{defin}
In words, $\Phi_{G,\rho}$ provides the best lower bound valid 
for all convolution powers of probability measures in 
$\mathcal S _{G,\rho}$. The following variant deals with finite weak-moments.

\begin{defin}[Fastest decay under weak-$\rho$-moment] \label{def2}
Let $G$ be a countable group.
Fix a function $\rho:[0,\infty)\ra[1,\infty)$ with $\rho(0)=1$.
Let $\widetilde{\mathcal S}_{G,\rho}$ be the set of all symmetric
probability   $\phi$ on $G$ with the properties that
$W(\rho,\phi)\le 2$. Set
$$\widetilde{\Phi}_{G,\rho}: n\mapsto \widetilde{\Phi}_{G,\rho}(n)= 
\inf\left\{
\phi^{(2n)}(e) : \phi \in \widetilde{\mathcal S}_{G,\rho}\right\}. $$
\end{defin}

\begin{rem} Since $\rho$ takes values in $[1,\infty)$, it follows that,
for any probability measure $\mu$ on $G$, 
we have $\sum\rho(|g|)\mu(g)\ge 1$ and $W(\rho,\mu)\ge 1$.  
In the definitions of $\Phi_{G,\rho}$  (resp. $\widetilde\Phi_{G,\rho}$), 
it is important to impose a uniform bound of the type 
$\sum \rho(|g|)\mu(g)\le 2$  (resp. $W(\rho,\mu)\le 2$) because   
relaxing this condition to $\sum \rho(|g|)\mu(g)<\infty $  
(resp. $W(\rho,\mu)< \infty$) would lead to a trivial  
$\Phi_{\rho,G}\equiv 0$ (resp. $\widetilde{\Phi}_{\rho,G}\equiv 0$). The
next remark indicates that, under natural circunstances, the choice of 
the particular constant $2$ in these definitions is unimportant. 
\end{rem}

\begin{rem} Assume that $\rho$ has the property that 
$\rho(x+y)\le C( \rho(x)+\rho(y))$. Under this natural condition 
\cite[Cor 2.3]{BSClmrw} shows that $\Phi_{G,\rho} $ and 
$\widetilde{\Phi}_{G,\rho}$ stay strictly positive. Further,
\cite[Prop 2.4]{BSClmrw} shows that, for any symmetric probability measure 
$\mu$ on $G$ such that $\sum \rho(|g|)\mu(g)<\infty$
(resp. $W(\rho,\mu)<\infty$), 
there exist constants $c_1,c_2$ (depending on $\mu$)
 such that
$$\mu^{(2n)}(e)\ge c_1 \Phi_{G,\rho}(c_2 n)$$
 (resp. $\mu^{(2n)}(e)\ge c_1 \widetilde{\Phi}_{G,\rho}(c_2n)$).
\end{rem}

Recall that a group $G$ is said to have polynomial (volume)  growth of 
degree $D$ if $V(n)=\#\{g\in G: |g|\le n\}\simeq n^D$. 
By a celebrated theorem of M. Gromov, a group $G$ such that $V(n)\le Cn^A$ 
for some fixed constants $C,A$ and all integers $n$  must be of polynomial growth of degree $D$ for some integers $D\in\{0,1,2,\dots\}$. In fact, 
Gromov's theorem states that such a group is virtually nilpotent, i.e., 
contains a nilpotent subgroup of finite index. See, e.g., \cite{delaH,Grom} and 
the references therein.
One of the most basic results proved in \cite{BSClmrw} is as follows.
\begin{theo}[\cite{BSClmrw}] \label{th-BSC}
Let $G$ have polynomial volume growth of 
degree $D$. For any $\alpha\in (0,2)$, let $\rho_\alpha(s)=(1+s)^\alpha$.
Then we have
$$\forall\,n\ge 1,\;\;\widetilde{\Phi}_{G,\rho_\alpha}(n)
\simeq n^{-D/\alpha}.$$
Moreover, if $\rho(s)\simeq [(1+s^2)\ell(1+s^2)]^\gamma$ with $\gamma\in (0,1)$
and $\ell$ smooth positive slowly varying at infinity 
with de Bruijn conjugate $\ell^\#$ then 
$$\forall\,n\ge 1,\;\;\widetilde{\Phi}_{G,\rho}(n)\simeq 
[n^{1/\gamma}\ell^{\#}(n^{1/\gamma})]^{-D/2}.$$
\end{theo}
\begin{rem}
This statement involves the notion of de Bruijn conjugate $\ell^{\#}$
of a positive 
slowly varying function $\ell$. 
We refer the reader to \cite[Theorem 1.5.13]{BGT}
for the definition and existence of the de Bruijn conjugate. Roughly speaking,
$\ell^{\#}$ is such that the inverse function of $s\mapsto s\ell(s)$ is 
$s\mapsto s\ell^{\#} (s)$. When $\ell $ is so slow that 
$\ell(s^a)\simeq \ell(s)$ for any $a>0$, then $\ell^{\#}\simeq 1/\ell$. 
For further results on de Bruijn conjugate, see \cite{BGT}.
\end{rem}

In the case when $\rho$ is slowly varying, \cite{BSClmrw} provides 
only partial results. In particular, the techniques of \cite{BSClmrw}
fail to give any kind of lower bound when $\rho(s)=\log( e+ s)^{ \epsilon}$ 
with $\epsilon\in (0,1]$ and for any $\rho$ that varies even slower than
these examples. The main goal of this work is to obtain 
detailed results in such  cases including the following theorem.
\begin{theo} \label{th-main}
Let $G$ have polynomial volume growth of 
degree $D$. For any $\epsilon>0$ we have
$$\widetilde{\Phi}_{G,\log^{\epsilon }}(n)\simeq 
\exp\left( - n^{1/(1+\epsilon)}\right)$$
where $\log^\epsilon$ stands for the function  
$\rho^{\log}_\epsilon(s)=[1+\log (1+s)]^\epsilon$.  Further, for any $k\ge 2$,
$$\widetilde{\Phi}_{G,(1+\log_{[k]})^{\epsilon }}(n)\simeq 
\exp\left( - n/(\log_{[k-1]}n)^\epsilon\right)$$
where $\log_{[k]}(x)=\log (1+\log_{[k-1]} x)$, $\log_{[0]} x=x$.
\end{theo} 
The upper bound on $\widetilde{\Phi}_{G,{\log}^\epsilon}$ is contained 
in \cite{BSClmrw,BSCsubord}. Developing techniques that provide a matching lower bound is the main contribution of this work. In \cite{BSClmrw}, 
$\widetilde{\Phi}_{G,{\log}^\epsilon}$ is bounded below by 
$\exp\left(- n^{1/\epsilon}\right)$ when  $\epsilon>1$ but \cite{BSClmrw}
provides no lower bounds at all  when $0<\epsilon\le 1$. As stated above, the 
present work provides 
sharp lower bounds under any iterated logarithmic weak-moment condition.

\begin{rem}The proof provided below for the lower bounds included in 
the statement of Theorem \ref{th-main} provides a much more precise result, 
namely, it provides some explicit measure $\mu_\rho$ which is a witness 
to the behavior of the infimum $\widetilde{\Phi}_{G,\rho}$ 
for the given $\rho$. We note that no such result is known for 
$\Phi_{G,\rho}$ in general and that even 
the precise behavior of $\Phi_{\mathbb Z,\rho_\alpha}$, $\alpha\in (0,2)$ 
is an open question.   
\end{rem}

To put our results in perspective, we briefly  comment on the classical 
case when $G=\mathbb Z$.  Let $\mu$ be a symmetric probability measure 
on $\mathbb Z$. The approximate local limit theorem of 
Griffin, Jain and Pruitt \cite{GJP} shows that if we set 
$$G(x)=\sum_{y: |y|> x}\mu(y),;\;
K(x)= x^{-2}\sum_{y:|y|\le x}|y|^2\mu(y) \mbox{ and } Q(x)= G(x)+K(x)$$
then, under the assumption that $\limsup_{x\ra \infty}G(x)/K(x)<\infty$,
$$\mu^{(2n)}(0)\simeq a_n^{-1} \mbox{ where } Q(a_n)=1.$$
This of course agrees with Theorem \ref{th-BSC} but fails to cover laws 
relevant to Theorem \ref{th-main} such that 
$$\mu(y)\simeq \frac{1}{(1+|y|)[\log (e+|y|)]^{1+\epsilon}]}$$
because, in such cases, $G$ dominates $K$. However, 
basic Fourier transform arguments show that
$$\widetilde{\Phi}_{\mathbb Z,{\log}^\epsilon}(n)\simeq \exp
\left( -n^{1/(1+\epsilon)}\right)$$
with the measure $\mu$ above being a witness of this behavior.

We close this introduction with a short description of the content of 
other sections. The main problem considered in this paper is the construction 
of explicite measures that satisfy (a) some given moment condition and 
(b) have a prescribed (optimal) behavior in terms of the probability 
of return after $n$ steps of the associate random walk. This is done by 
using subordination techniques based on Bernstein functions.

Section 2 describes how the notion of Bernstein function and the 
associated subordination techniques lead to a variety of 
explicit examples of probability 
measures whose iterated convolutions can be estimated precisely when 
the underlying group has polynomial volume growth. See Theorems 
\ref{th-phipsi1}-\ref{th-phipsi2}.  

Section 3 describes assorted results for measures 
that are supported only on powers of the generators when a given generating 
set has been chosen.

Section 4 develops a set of Pseudo-Poincar\'e inequalities adapted to 
random walks driven by symmetric probability measures that only have 
very low moments. These Pseudo-Poincar\'e inequalities are essential to the 
arguments developed in this paper. 

Section 5 contains the main result of this article, Theorem \ref{th-mainl} 
of which Theorem \ref{th-main} is an immediate corollary.  

The entire paper is written in the natural context of discrete time random 
walks. Well-known general techniques allow to translate the main results 
in the context of continuous time random walks. 
y
\section{The model case provided by subordination}
\setcounter{equation}{0}

Recall that a Bernstein function is a function 
$\psi\in \mathcal C^\infty((0,\infty))$ such that $\psi\ge 0$ and 
$(-1)^k \frac{d^k\psi}{dt^k}\le 0$. 
A classical result asserts that a function $\psi$ is a 
Bernstein function if and only if there are reals $a,b\ge0$ and a measure 
$\nu$ on $(0,\infty)$ satisfying $\int_0^\infty \frac{td\nu(t)}{1+t}<\infty$ such that $\psi(s)= a+bs +\int_0^\infty (1-e^{-st})d\nu(t)$. Set 
\begin{equation}\label{def-c}
c(\psi,1)=b+\int_0^\infty te^{-t}d\nu(t),\;\; c(\psi,n)=\frac{1}{n!}\int_0^\infty t^ne^{-t}d\nu(t), \;n>1.
\end{equation}
If $\psi$ is a Bernstein function satisfying $\psi(0)=0$, $\psi(1)=1$ and $K$ is a Markov kernel then
$$K_\psi= \sum_1^\infty c(\psi,n)K^n$$
is also a Markov kernel.  
Further, one can understand $K_\psi$ as given by $K_\psi=I-\psi(I-K)$. See 
\cite{BSCsubord} for details.
Similarly, if $\phi $ is a probability measure on a group $G$, set
$$\phi_\psi=\sum c(\psi,n)\phi^{(n)}.$$
This is a probability measure which we call the $\psi$-subordinate of $\phi$.

Recall that a complete Bernstein function is a function
$\psi\in \mathcal C^\infty((0,\infty))$ such that 
\begin{equation} \psi(s)=s^2\int_0^\infty e^{-ts}g(t)dt 
\end{equation} 
where $g$ is a Bernstein function (complete Bernstein functions are Bernstein 
function).  A comprehensive book treatment of the theory of Bernstein functions
is \cite{SSV}. See also \cite{Jacob}. 

\begin{exa} The most basic examples of  complete Bernstein functions are 
$\psi(s)=s^\alpha$, $\alpha\in (0,1)$, and $\psi(s)= \log_2(1+ s)$. 
A less trivial example of interest to us is 
$$\psi(s)= \frac{1}{[\log_2 (1+ s^{-1/\alpha})]^{\beta\alpha }},\; 
0<\beta\le 1\le  \alpha<\infty.$$ 
The choice of the base $2$ logarithm in this definition insures that
the additional property $\psi(1)=1$ holds true.  
If we define $\log_{2,k}$ by setting $\log _{2,1}(s)=\log_2(1+s)$
and $\log_{2,k}(s)=\log_{2,1}(\log_{2,k-1}(s))$, $k>1$, then 
the function
$$\psi(s)= \frac{1}{[\log_{2,k} (s^{-1/\alpha})]^{\beta\alpha} },\;
0<\beta\le 1 \le \alpha<\infty$$
is also a complete Bernstein function. 
\end{exa}

The following two results from \cite{BSCsubord} will be very
 useful for our purpose. For a comprehensive treatment of the theory of 
function of slow and regular variation, see \cite{BGT}.

\begin{theo}[{\cite[Theorems 2.5--2.6]{BSCsubord}}] \label{th-cpsi}
Assume that $\psi_1:(0,\infty)\ra (0,\infty)$ has a positive 
continuous derivative and satisfies $\psi_1(0^+)=0$. Assume further that 
$x\mapsto \psi_1(x)$
and $x\mapsto x\psi_1'(x)$ are slowly varying at $0^+$ and that
$$\psi'_1(s)\sim \frac{1}{s\ell(1/s)}$$ where $\ell$ 
is slowly varying at infinity. Then there exists
a positive constant $a$ and a complete Bernstein function $\psi$ such that $\psi\sim a \psi_1$, $\psi'\sim a\psi_1'$ at $0^+$ and $\psi(1)=1$. Further
$$c(\psi,n)\simeq \frac{1}{(1+n)\ell(1+ n)}.$$
\end{theo} 
\begin{theo}[{\cite[Theorems 3.3--3.4]{BSCsubord}}]
Let $G$ be a finitely generated group of polynomial volume growth. Let 
$\phi$ be a finitely supported symmetric probability measure with $\phi(e)>0$ 
and generating support.  Let $\psi$ be a Bernstein function with $\psi(0)=0$, $\psi(1)=1$. Assume that 
$$\psi(s)\simeq 1/\theta(1/s)$$
with $\theta$ positive increasing slowly varying at infinity. \begin{enumerate}
\item Assume that
the rapidly varying function $\theta^{-1}$ (the inverse function of $\theta$)
satisfies $$\log \theta^{-1} (u)\simeq u^{\gamma} \kappa (u)^{1+\gamma}$$ 
with $\gamma\in (0,\infty)$ and $\kappa$ slowly varying at infinity. 
Then the $\psi$-subordinate $\phi_\psi$ of $\phi$ satisfies
$$-\log( \phi_\psi^{(n)}(e))\simeq n^{\gamma/(1+\gamma)}/
\kappa^{\#}(n^{1/(1+\gamma)})$$
where $\kappa^{\#}$ is the de Bruijn conjugate of $\kappa$.
\item Assume that the function  $\kappa = \theta\circ \exp$ is slowly varying and  satisfies  $s\kappa^{-1}(s)\simeq \kappa^{-1}(s)$ at infinity.
Then the $\psi$-subordinate $\phi_\psi$ of $\phi$ satisfies
$$-\log( \phi_\psi^{(n)}(e))\simeq  n/
\kappa(n).$$
\end{enumerate}
\end{theo}
\begin{exa} Fix an integer $k\ge 1$ and parameters $\alpha,\beta$, 
$0<\beta\le 1 \le \alpha<\infty$. Consider the complete Bernstein function  
$$\psi(s)= \frac{1}{\theta(1/s)}=\frac{1}{[\log_{2,k} (s^{-1/\alpha})]
^{\beta\alpha} }.$$
A simple computation yields
$$\psi'(s)=   \frac{\beta}{s(1+s^{1/\alpha})\log_{2,1}(s^{-1/\alpha})\cdots 
\log_{2,k-1}(s^{-1/\alpha}) [\log_{2,k} (s^{-1/\alpha})]^{1+\beta\alpha}}.$$
It follows that
$$c(\psi,n)\simeq \frac{1}{(1+n) (1+\log_{[1]}n)
\cdots (1+\log_{[k-1]}n)(1+\log_{[k]}n)^{1+\beta\alpha}}.$$
Here $\log_{[m]}n=\log(1+ \log_{[m-1]}n)$ with $\log_{[1]}n=\log (1+n)$.
Further if $k=1$, we have  
$\log \theta ^{-1} (u) \simeq u ^{1/\alpha \beta} $
and it follows that $$\phi_\psi^{(n)}(e)\simeq \exp\left(-n^{1/(1+\alpha\beta)}
\right).$$
In the case where $k>1$, we have
$\kappa(s)=\theta\circ \exp (s)  \simeq [\log_{[k-1]}(s)]^{\alpha\beta}$
and we obtain 
$$\phi_\psi^{(n)}(e)\simeq 
\exp\left(-n/ [\log_{[k-1]}(n)]^{\alpha\beta}\right).$$
\end{exa}

For later purpose, we need the information contained in the following Theorem
which is an easy corollary of Theorem \ref{th-cpsi} 
and the Gaussian bounds of \cite{HSC}.
\begin{theo} \label{th-phipsi1}
Let $\psi_1$ and $\psi$ be as in {\em Theorem \ref{th-cpsi}}
with $$\psi'_1(s)\sim \frac{1}{s\ell(1/s)}\;\mbox{ at } 0^+,$$ where $\ell$ 
is slowly varying at infinity.  
Let $G$ be a finitely generated group of polynomial volume growth of degree $D$. 
Let 
$\phi$ be a finitely supported symmetric probability measure with $\phi(e)>0$ 
and generating support. 
Then there are constants $c,C\in (0,\infty)$ such that the probability measure 
$\phi_\psi$ satisfies
$$\forall\,x\in G,\;\; \frac{c}{(1+|x|)^{D}\ell(1+|x|^2)}\le
\phi_\psi(x) \le \frac{C}{(1+|x|)^{D}\ell(1+|x|^2)}.$$
\end{theo}
\begin{proof} By \cite{HSC},
there are constants $c_i$, $1\le i\le 4$,
such that for each $x, n$ such that $\phi^{(2n)}(x)\neq 0$,
$$
c_1 n^{-D/2} \exp\left(- c_2\frac{|x|^2}{n}\right) \le 
\phi^{(n)}(x)\le c_3 n^{-D/2} \exp\left(- c_4\frac{|x|^2}{n}\right).$$ 
By Definition and Theorem \ref{th-cpsi},  $\phi_\psi^{(n)}(x)$ is bounded 
above and below 
$$ \sum_1^\infty  \frac{c}{(1+n)\ell(1+n)} \phi^{(n)}(x)$$
(with different constants $c$ in the upper and lower bound). Break this sum 
into two parts $S_1,S_2$ with $S_1$ being the sum over $n\ge |x|^2$. We have
$$S_1 \simeq \sum_{n\ge |x|^2} \frac{1}{(1+n)^{1+D/2}\ell(1+n)} \simeq  
\frac{1}{(1+|x|^2)^{D/2}\ell(1+|x|^2)}.$$
which already proves the desired lower bound.
Similarly, for $S_2$,  note that
$$\frac{(1+|x|)^{2+D}\ell(1+|x|^2)}{(1+n)^{1+D/2}\ell(1+n)}\le C 
\left(\frac{1+|x|^2}{1+n}\right)^A$$
for some $A>0$. Further, for each $k$,  there are at most   $2|x|^2/k^2$ positive integers $n$ 
such that
$k-1< |x|^2/n\le k$. Hence, we obtain
\begin{eqnarray*}
S_2 &\le & \frac{C'}{(1+|x|)^{2+D}\ell(1+|x|^2)}\sum _k (1+|x|^2)(1+k)^{A-2}e^{-c_4 (k-1)}\\
&\le & \frac{C''}{(1+|x|)^D\ell(1+|x|^2)}.
\end{eqnarray*}
Together with the estimate already obtained for $S_1$, this gives 
the desired upper bound on 
$\phi_\psi^{(n)}(x)$.
\end{proof}

The following statement put together the results gathered above while 
emphasizing the point of view of the construction of a model with a 
prescribed behavior.

\begin{theo} \label{th-phipsi2}
Let $G$ be a finitely generated group with polynomial 
volume growth of degree $D$. Let $\phi$ be a finitely supported symmetric
probability measure with $\phi(e)>0$ and generating support.
Let $\ell$ be a continuous positive slowly  
varying function at infinity such that $\int_0^1 \frac{ds}{s\ell(1/s)}<\infty$.
Then there exists a complete Bernstein function $\psi$ with $\psi(0)=0$, 
$\psi(1)=1$ such that:
\begin{itemize}
\item  $\psi(s)\sim a\int_0^{s}\frac{dt}{t\ell(1/t)}$ for some 
constant $a>0$;
\item $c(\psi,n)\simeq \frac{1}{(1+n) \ell (1+n)}$;
\item $\phi_\psi(x)\simeq [(1+|x|)^D\ell(1+|x|^2)]^{-1}$;  
\end{itemize}
Further, if we set $\theta(s)= 1/\int_0^{1/s} \frac{dt}{t\ell(1/t)}$, the following holds:
\begin{itemize} 
\item If $\log \theta^{-1} (u) \simeq u^\gamma \kappa(u)^{1+\gamma}$
with $\gamma\in (0,\infty)$ and $\kappa$ slowly varying at infinity, then we have
$$\phi_\psi^{(n)}(e)\simeq \exp\left(
- n^{\gamma/(1+\gamma)}/\kappa^{\#}(n^{1/(1+\gamma)})\right)$$
where $\kappa^{\#}$ is the de Bruijn conjugate of $\kappa$.
\item If $\kappa=\theta \circ \exp$ is slowly varying and satisfies 
$s\kappa^{-1}(s)\simeq \kappa^{-1}(s)$ then
$$\phi_\psi^{(n)}(e)\simeq \exp\left(- n/\kappa(n)\right).$$ 
\end{itemize}
\end{theo}

\begin{exa} Let $G$ be a finitely generated group with polynomial volume growth 
of degree $D$. 
Then, for any $\delta>0$, 
there exists a symmetric probability measure $\phi_\delta$
such that 
$$\phi_\delta(x)\simeq \frac{1}{(1+|x|)^D[\log (e+|x|)]^{1+\delta}}$$
and 
$$\phi_\delta^{(n)}(e)\simeq \exp\left(- n^{1/(1+\delta)}\right).$$
Also, for any $\delta>0$ and integer $k\ge 1$, there exists a symmetric probability measure $\phi_{k,\delta}$ such that
$$\phi_{k,\delta}(x)\simeq \frac{1}{(1+|x|)(1+\log_{[1]} |x|)\cdots(1+\log_{[k-1]}|x|)(1+\log_{[k]}|x|)^{1+\delta}}$$
and
$$\phi_{k,\delta}^{(n)}(e)\simeq \exp \left(-n/(\log_{[k-1]} n)^{1/\delta}\right)
.$$
Further, for any $k\ge 1$ 
and $\delta>0$,   the comparison results of \cite{PSCstab} imply that
any symmetric probability measure $\varphi$ with the property that
$$\forall\, f\in L^2(G),\;\;\mathcal E_\varphi(f,f)\le \mathcal E_{\phi_{k,\delta}}(f,f)
\le C\mathcal E_\varphi(f,f)$$ satisfies $\varphi^{(2n)}(e)\simeq \phi_{k,\delta}^{(2n)}(e)$.  See  (\ref{def-Emu}) for a definition of the Dirichlet form 
$\mathcal E_\mu$ associated with a symmetric probability measure $\mu$.
\end{exa}

\section{Measures supported on powers of generators}
\setcounter{equation}{0}

In \cite{SCZ-nil}, the authors introduced the study 
of random walks driven by measures supported on the powers of given generators.
Namely, given a group $G$ equipped with a  generating $k$-tuple 
$(s_1,\dots,s_k)$, fix a $k$-tuple of probability measures $(\mu_i)_1^k$, each
$\mu_i$ being a probability measure on $\mathbb Z$, and set
\begin{equation}\label{def-mu}
\mu(g)=k^{-1}\sum_1^k\sum_{n\in \mathbb Z}\mu_i(n)\mathbf 1_{s_i^n}(g).
\end{equation}
In \cite{SCZ-nil}, special attention is given to the case when the 
$\mu_i$ are symmetric power laws. Here, we focus on the case when the $\mu_i$ are symmetric, all equal and are of the type
$$\mu_i(n)=\phi(n)  \simeq \frac{1}{(1+n)\ell(1+n)}$$
where $\ell$ is increasing and slowly varying. 
Obviously, we require here that $\sum_1^\infty [n\ell(n)]^{-1}<\infty$.

The following statement is a special case of \cite[Theorem 5.7]{SCZ-nil}. 
It provides a key comparison between the Dirichlet forms of 
measures supported on power of generators and associated measures that
are radial with respect to the word-length. In this form, this 
result holds  only under the hypothesis that the group $G$ is nilpotent. 
Recall that 
the Dirichlet form $\mathcal E_\mu$ associated with a symmetric probability measure $\mu$ is the quadratic form on $L^2(G)$ given by
\begin{equation}\label{def-Emu}
\mathcal E_\mu (f,f)=\frac{1}{2}\sum_{x,y}|f(xy)-f(x)|^2\mu(y).
\end{equation}

\begin{theo} \label{th-comp}
Let $G$ be a nilpotent group equipped with a generating $k$-tuple 
$S=(s_1,\dots,s_k)$. Let $|\cdot|$ be the corresponding word-length 
and $V$ be the associated volume growth function. Fix a continuous increasing
function $\ell: [0,\infty)\ra (0,\infty)$ which is slowly varying at infinity.
Assume that 
\begin{equation}
\sum_{g\in G}\frac{1}{V(1+|g|)\ell(1+|g|)}<\infty.
\end{equation}
Consider the probability measures $\nu$ and $\mu$ defined by
\begin{equation}
\nu(g)= \frac{c}{V(1+|g|)\ell(1+|g|)}, \;\;c^{-1}= \sum\frac{1}{V(1+|g|)\ell(1+|g|)},\end{equation}
and 
\begin{equation}
\mu(g)= k^{-1}\sum_1^k\sum_{n\in \mathbb Z} \frac{b \mathbf 1 _{s_i^n}(g)}{(1+n)\ell(1+n)}
,\;\;b^{-1}=\sum_{\mathbb Z}\frac{1}{(1+n)\ell(1+n)}.
\end{equation} 
Then there exists a constant $C$ such that
$$\mathcal E_\mu\le C\mathcal E_\nu.$$
\end{theo}
\begin{proof} We apply \cite[Theorem 5.7]{SCZ-nil} with $\phi=\ell$ and 
$\|\cdot\|=|\cdot|$ (in the notation of \cite{SCZ-nil}, 
this corresponds to having a weight system $\mathfrak w$ generated 
by $w_i=1$ for all $i=1,\dots,k$, the simplest case). Referring to the 
notation used in \cite{SCZ-nil}, 
because of the choice 
$\|\cdot\|=|\cdot|$, we have $F_{c_1}(r)=r$, $F_{h_i}(r)= r^{m_i}$ 
where $m_i\ge 1$ ($m_i$ is an integer which describes the position of 
the generator $s_i$ in the lower central series of $G$, modulo torsion). 
Having made these observations, the stated 
result follows from \cite[Theorem 5.7]{SCZ-nil} by inspection.
\end{proof}

\begin{rem} Note that, in the context of Theorem \ref{th-comp} 
and for any positive function $\ell$ that is  
slowly varying at infinity, the conditions 
$$(a)\; \sum_1^\infty\frac{1}{n\ell(n)}<\infty;\;\; 
(b)\;\sum_1^\infty\frac{1}{n\ell(n^2)}<\infty;\;\; (c) 
\;\sum_{g\in G} \frac{1}{
V(|g|)\ell(1+|g|)},$$
are equivalent. To see that $(a)$ and $(c)$ are equivalent, note that
$$\sum_g\frac{1}{V(|g|)\ell(1+|g|)}\simeq 
\sum_k \frac{V(k)-V(k-1)}{\ell(1+ k)
(1+k)^D\ell(1+ k)}$$
and use Abel summation formula to see that this implies
$$\sum_g\frac{1}{V(|g|)\ell(1+|g|)}\simeq  
\sum_k \frac{k^D}{
(1+k)^{D+1}\ell(1+k)}\simeq \sum_k \frac{1}{k\ell(k)}.$$ 
\end{rem}

The following statement illustrates one of the basic consequences 
of this comparison theorem.

\begin{theo} \label{th-muell}
Let $\ell: [0,\infty)\ra [0,\infty)$ be continuous increasing, 
slowly varying at infinity, and such that 
$\int_0^{1/s} \frac{dt}{t\ell(1/t)} <\infty$. Set
$\theta(s)= 1/\int_0^{1/s} \frac{dt}{t\ell(1/t)}$.
Let $G$ be a finitely generated nilpotent group equipped with a 
generating $k$-tuple $S=(s_1,\dots,s_k)$. Let $\mu$ be the symmetric 
probability measure on $G$  defined by
$$\mu(g)=k^{-1}\sum_1^k\sum_{n\in \mathbb Z} 
\frac{c \mathbf 1_{s_i^n}(g)}{(1+|n|)\ell(1+|n|^2)}.$$
\begin{itemize} 
\item If $\log \theta^{-1} (u) \simeq u^\gamma \kappa(u)^{1+\gamma}$
with $\gamma\in (0,\infty)$ and $\kappa$ slowly varying at infinity, then we have
$$\mu^{(n)}(e)\simeq \exp\left(
- n^{\gamma/(1+\gamma)}/\kappa^{\#}(n^{1/(1+\gamma)})\right)$$
where $\kappa^{\#}$ is the de Bruijn conjugate of $\kappa$.
\item If $\kappa=\theta \circ \exp$ is slowly varying and satisfies 
$s\kappa^{-1}(s)\simeq \kappa^{-1}(s)$ then
$$\mu^{(n)}(e)\simeq \exp\left(- n/\kappa(n)\right).$$ 
\end{itemize}
\end{theo}
\begin{proof} The lower bounds follow from Theorems \ref{th-phipsi2} 
and \ref{th-comp}, together with \cite[Theorem 2.3]{PSCstab}. To prove the 
upper bounds,
forget all but one non-torsion generator, say $s_1$, and  use 
\cite[Theorem ]{PSCstab} to compare with the corresponding one-dimensional 
random walk on $\{s_1^n: n\in \mathbb Z\}$. 
\end{proof}
\begin{rem} Assume that $\mu$ is given by (\ref{def-mu}) with possibly different
$\mu_i$ of the form $\mu_i(m) \simeq 1/[(1+|m|)\ell_i(1+|m|^2)]$. Let $\ell,\theta$ be as in Theorem \ref{th-muell}. If there exists $i\in \{1,\dots, k\}$ such that $s_i$ is not torsion and $\ell_i\le C\ell$ then $\mu^{(2n)}(e)$ 
can be bounded above by the convolution power $\phi^{(2n)}(0)$
of the one dimensional 
symmetric probability measure $\phi(m)= c/(1+|m|)\ell(1+|m|^2)$. If we assume that
for all $i\in \{1,\dots,k\}$ such that $s_i$ is not torsion we have 
$\ell_i\ge c\ell$ then we obtain a lower bound for 
$\mu^{(2n)}(e)$ in terms of $\phi^{(2n)}(0)$.
\end{rem}

\section{Pseudo-Poincar\'e inequality}
In \cite{SCZ-nil}, the authors proved and use new (pointwise) 
pseudo-Poincar\'e inequalities adapted to spread-out probability measures.
These pseudo-Poincar\'e inequalities are proved for measures of type 
(\ref{def-mu}) and involve the truncated second moments of the one 
dimensional probability measures $\mu_i$. More precisely, fix $s\in G$ and let $\phi$ be a symmetric probability measure on $\mathbb Z$. 
It is proved in \cite{SCZ-nil} that, if we set
$$\mathcal E_{s,\phi}(f,f)=
\frac{1}{2}
\sum_{x\in G}\sum_{n\in \mathbb Z}|f(xs^n)-f(x)|^2\phi(n),\;\;\mathcal G_\phi(r)
=\sum_{|n|\le r}|n|^2\phi(n),$$  
and assume that there exists a constant $C$ such that 
$\phi(n)\le C\phi(m)$ for all $|m|\le |n|$, then it holds that
\begin{equation}
\sum_{x\in G} |f(xs^n)-f(x)|^2\le C_\phi(\mathcal G_\phi(|n|))^{-1}|n|^2
\mathcal E_{s,\phi}(f,f).
\end{equation}
Under the same notation and hypotheses, set 
$\mathcal H_\phi(r)= \sum_{|n|>r}\phi(n)$.  Then we claim that\begin{equation}
\sum_{x\in G}|f(xs^n)-f(x)|^2\le C'_\phi(\mathcal H_\phi(|n|))^{-1}
\mathcal E_{s,\phi}(f,f).
\end{equation} 
Indeed, write 
$$|f(xs^n)-f(x)|^2\le 2(|f(xs^n)-f(xs^m)|^2+|f(xs^m)-f(x)|^2).$$
Note note that the set $\{m: |n-m|\le |m|\}$ contains $\{m:m\ge n\}$ 
if $n$ is positive and $\{m: m\le n\}$ if $n$ is negative. 
Multiply both sides of the displayed  inequality above
 by $\phi(m)$ and sum over $x\in G$ 
and  $m$ such that $|n-m|\le |m|$ to obtain 
\begin{eqnarray*}
\lefteqn{\left(\sum_{x\in G}|f(xs^n)-f(x)|^2\right)\mathcal H_\phi(|n|)}&&\\
&\le & 
4\sum_{x\in G}\sum_{|m|\ge |n|}(|f(xs^n)-f(xs^m)|^2\phi(m) +|f(xs^m)-f(x)|^2\phi(m))\\
&\le & 4\sum_{x\in G}\sum_{m\in \mathbb Z}(|f(xs^{n-m})-f(x)|^2\phi(m) +
|f(xs^m)-f(x)|^2\phi(m))\\
\\
&\le & 4C\sum_{x\in G}\sum_{m\in \mathbb Z}(|f(xs^{n-m})-f(x)|^2\phi(n-m) +
|f(xs^m)-f(x)|^2\phi(m))\\
&= & 16C \mathcal E_{s,\phi}(f,f).
\end{eqnarray*}
Putting together this simple computation and the earlier results from 
\cite{SCZ-nil}, we can state the following theorem.
\begin{theo} \label{th-PP}
Let $\phi$ be a symmetric probability measure on 
$\mathbb Z$ such that
there exists a constant $C$ for which, for all $|m|\le |n|$, 
$\phi(n)\le C \phi(m)$. There exists a constant $C_\phi$ such that, for any group $G$ and any $s\in G$, we have
$$\forall\,n, \;\;\sum_{x\in G}|f(xs^n)-f(x)|^2\le C_\phi 
\min\left\{\frac{1}{\mathcal H_\phi(|n|)},\frac{|n|^2}{\mathcal G_\phi(|n|)}
\right\} \mathcal E_{s,\phi}(f,f).$$
\end{theo}
\begin{rem}
 If $\phi$ is regularly varying of index $\alpha\in (-3,-1)$, then 
$\mathcal H_\phi(r)\simeq  r^{-2}\mathcal G_\phi(r)$. 
If $\phi $ is regularly varying of index $\alpha<-3$ then 
$\mathcal H_\phi(r)$ is much smaller than $r^{-2}\mathcal G_\phi(r)$.
When $\phi$ is regularly varying of index $-1$ then 
$\mathcal H_\phi(r)$ is much larger than $r^{-2}\mathcal G_\phi(r)$.
\end{rem}

\begin{cor}\label{cor-PP1}
Let $\phi$ be a symmetric probability measure on $\mathbb Z$ such that
there exists a constant $C$ for which, for all $|m|\le |n|$, 
$\phi(n)\le C \phi(m)$. 
Let $G$ be a finitely generated nilpotent group equipped with a 
generating $k$-tuple $S=(s_1,\dots,s_k)$. Let $\mu$ be the symmetric 
probability measure on $G$  defined by
$$\mu(g)=k^{-1}\sum_1^k\sum_{n\in\mathbb Z} 
\phi(n)\mathbf 1_{s_i^n}(g).$$
Then there are constants $C_1,C_2$ such that, for all $g\in G$, we have
$$\forall\, f\in L^2(G), \sum_{x\in G}|f(xg)-f(x)|^2\le 
C_1 
\min\left\{\frac{1}{\mathcal H_\phi(C_2|g|)},\frac{|g|^2}{\mathcal G_\phi(C_2|g|)}
\right\} \mathcal E_{\mu}(f,f).$$
\end{cor}
\begin{proof} 
Apply \cite[Theorem 2.10]{SCZ-nil} in the 
simplest case when the weight system $\mathfrak w$ is generated by constant weights $w_i=1$, $1\le i\le k$, so that the corresponding length function on $G$ is 
just the  word-length $g\mapsto |g|$.  This result yields the existence of 
a constant $C_0$, an integer $M$ and a 
sequence $(i_1,\dots i_M)\in \{1,\dots, k\}^M$ such that any element 
$g\in G$ can be written in the form
$$g=\prod_{j=1}^M s_{i_j}^{x_j} \mbox{ with } |x_j|\le C_0|g|.$$
Further, by construction, for each $i\in \{1,\dots,k\}$,
$$\mathcal E _{s_i,\phi}\le k \mathcal E_\mu .$$
Hence, the stated Corollary follows easily from a finite telescoping sum 
argument and Theorem \ref{th-PP}.
\end{proof}

\begin{theo}\label{th-PPpol}
Let $\ell: [0,\infty)\ra [0,\infty)$ be continuous increasing, 
slowly varying at infinity, and such that 
$\int_0^{1/s} \frac{dt}{t\ell(1/t)} <\infty$. Set
$\theta(s)= 1/\int_0^{1/s} \frac{dt}{t\ell(1/t)}$. 
Let $G$ be a finitely generated group with word-length $|\cdot|$ and  polynomial 
volume growth of degree $D$. Let $\varphi$ 
be a symmetric probability measure on $G$ such that
$$\varphi(g)\simeq \frac{1}{
(1+|g|)^D\ell(1+|g|^2)}.$$
Then, there exists a constant $C$ such that for any $g\in G$ and any 
$f\in L^2(G)$,
$$\sum_{x\in G}|f(xg)-f(x)|^2\le C \theta (1+|g|^2)\mathcal E_\varphi(f,f).$$
\end{theo}
\begin{proof} As a key first step in the proof of this theorem, consider 
the special case when $G$ is a finitely generated nilpotent group equipped 
with a generating $k$-tuple $S=(s_1,\dots,s_k)$. In this case, the theorem 
follows from Corollary \ref{cor-PP1} 
and Theorem\ref{th-comp} by inspection after noting that
$$\mathcal H _\phi (r) \simeq 1/\theta (r^2).$$ 

Next, consider the general case when $G$ has 
polynomial volume growth of degree $D$. Then, by Gromov's theorem \cite{Grom},  
$G$ contains a finitely generated nilpotent group $G_0$  of finite index in $G$.
Fix finite symmetric generating sets in $G$ and $G_0$. Let $|\cdot|$ 
be the word-length in $G$ and $\|\cdot\|$ be the word-length in $G_0$. 
It is well-known that, for any $g_0\in G_0\subset G$, we have
$\|g_0\|\simeq |g_0|$. 

Let $A,B$ be finite sets of coset representatives for  $G_0\backslash G$
and $G/G_0$, respectively. Fix $g\in G$ 
and write $g=g_0b$, $g_0\in G_0$, $b\in B$. 
Observe that 
$G=\{x=ax_0: a\in A, x_0\in G_0\}$. Hence, for any $f\in L^2(G)$, we can write
$$\sum_{x\in G}|f(xg)-f(x)|^2=
\sum_{a\in A}\sum_{x_0\in G_0}|f(ax_0g_0b)-f(ax_0)|^2.$$
Applying the result already proved for nilpotent groups to $G_0$ and the 
functions $f_a: G_0\ra \mathbb R, f_a(x_0)=f(ax_0)$, $a\in A$, we obtain
$$\sum_{x_0\in G_0}|f(ax_0g_0)-f(ax_0)|^2\le \frac{C \theta(1+\|g_0\|^2)}{2}
\sum_{x_0,y_0\in G_0}\frac{|f(ax_0y_0)-f(ax_0)|^2}
{\ell(1+\|y_0\|^2)(1+\|y_0\|)^D}.$$
Summing over $a\in A$ and using the fact that $\|g_0\|\simeq |g_0|$ easily yield
$$\sum_{x\in G}|f(xg_0)-f(x)|^2\le C \theta(1+|g_0|^2)\mathcal E_\varphi(f,f).$$
Since we trivially have 
$$\forall\,b\in B,\;\;\sum_{x\in G}|f(xb)-f(x)|^2\le C \mathcal E_\varphi(f,f),$$
the desired result follows.
\end{proof}

\section{Probability of return lower bounds under weak-moment conditions}

In this section we use the results obtained in earlier Sections 
together with \cite[Theorem 2.10]{BSClmrw} to prove our main theorem, 
Theorem \ref{th-mainl}. Note that Theorem \ref{th-main} stated in the 
introduction is  an immediate corollary of this more general result.  

\begin{theo} \label{th-mainl}
Let $G$ be a finitely generated group with word-length $|\cdot|$ and
polynomial volume growth of degree $D$.  Let $\ell:[0,\infty)\ra [0,\infty)$ 
be a positive continuous increasing function
which is slowly varying at infinity and satisfies
$\int_1^\infty \frac{dt}{t\ell(t)}<\infty$. Set 
$\theta(s)=1/\int_0^{1/s}\frac{dt}{t\ell(1/t)}$ and 
$\theta_2(s)=\frac{1}{2}\theta(s^2)$. Let $\varphi$ 
be a symmetric probability measure such that
$$\varphi(g)\simeq \frac{1}{(1+|g|)^D\ell(1+|g|^2)}.$$
Then we have
$$\widetilde{\Phi}_{G,\theta_2} (n)\simeq \varphi^{(n)}(e).$$
\end{theo}

The proof of this result is based on a simple special case of 
\cite[Theorem 2.10]{BSClmrw}. For clarity and 
the convenience of the reader, we state the precise statement we need. Abusing notation, if $\varphi_1$ is a probability measure and $c_\alpha(n)$ 
is defined by $1-(1-x)^\alpha=\sum c_\alpha(n)x^n$, $x\in [-1,1]$, 
$\alpha\in (0,1)$, we call $\varphi_\alpha=\sum c_\alpha(n)\varphi_1^{(n)}$ the 
$\alpha$-subordinate of $\varphi_1$.

\begin{theo}[See {\cite[Theorem 2.10]{BSClmrw}}] \label{th-W}
Let $G$ be a finitely generated group with word-length $|\cdot|$.
Let $\varphi_1$ be a symmetric probability measure on a group $G$ and $\delta$ be a positive increasing function with $\delta(0)=1$.  Assume that
, for any $g\in G$ and $f\in L^2(G)$,
$$\sum_{x\in G}|f(xg)-f(x)|^2\le C\delta(|g|)^2\mathcal E_{\varphi_1}(f,f).$$
Fix $\alpha\in (0,1)$. Let $\mu$ be a symmetric measure on $G$ satisfying the weak moment condition
$$W(\delta^{2\alpha},\mu)=\sup_{s>0}\left\{s\mu(\{g: \delta(|g|)^{2\alpha}>s\})\right\}<\infty.$$
Then, for all $f\in L^2(G)$, 
$$\mathcal E_\mu(f,f)\le C_\alpha C W(\delta^{2\alpha},\mu) \mathcal E_{\varphi_\alpha}(f,f)$$
where $\varphi_\alpha$ is the $\alpha$-subordinate of $\varphi_1$. In particular,
$$\mu^{(2n)}(e)\ge c\varphi_\alpha^{(2Nn)}(e).$$
\end{theo}
\begin{proof} This is a special case of \cite[Theorem 2.10]{BSClmrw}. Referring to the notation used in  \cite[Theorem 2.10]{BSClmrw}, the operator $A$ is taken to be $Af=f*(\delta_e-\varphi_1)$, the function  $\psi$ is simply 
$\psi(s)=s^\alpha$ so that $\omega(s)=\Gamma(2-\alpha)^{-1}s^{1-\alpha}$. 
It follows that the function $\rho$ satisfies $\rho (s)\simeq 1+s^{2\alpha}$. 
Note that,
by definition, $\|\psi(A)^{1/2}f\|_2^2=\mathcal E_{\phi_\alpha}(f,f)$.
The last statement in the theorem follows from \cite{PSCstab}.
\end{proof}

\begin{proof}[Proof of {Theorem \ref{th-mainl}}] To prove that 
$n\mapsto \widetilde{\Phi}_{G,\theta_2}(n)$ is controlled from above by  
$n\mapsto \varphi^{(2n)}(e)$, 
it suffices to show that 
$\varphi$ has a finite weak-$\theta_2$-moment. For $s\ge 1$, write
\begin{eqnarray*}
\varphi(\{g: \theta_2(|g|)>s\})&=& \sum _{|g|\ge \theta_2^{-1}(s) }\frac{1}{(1+|g|)^D\ell(1+|g|^2)}\\
&\simeq &\sum _{k\ge \theta_2^{-1}(s)} \frac{V(k)-V(k-1)}{(1+k)^D \ell(1+k^2)}\\
&\simeq &\sum _{k\ge \theta_2^{-1}(s)} \frac{1}{(1+k) \ell(1+k^2)}\\
&\simeq & 1/\theta_2 (\theta_2^{-1}(s))\simeq 1/s.
\end{eqnarray*}
This shows that $W(\theta_2,\varphi)<+\infty$. 
By \cite[Proposition 2.4]{BSClmrw},
this implies that there exist $N,C$ such that, for all $n$,  
$\widetilde{\Phi}_{G,\theta_2}(Nn)\le C\varphi^{(2n)}(e)$. 

The more interesting statement is the bound 
$$ \widetilde{\Phi}_{G,\theta_2}(n) \ge c \varphi^{(2Nn)}(e) .$$
Let $\phi$ be a symmetric finitely supported probability measure on 
$G$ with generating support and $\phi(e)>0$.  
Using the basic hypothesis regarding the function $\ell$ and Theorems 
\ref{th-cpsi} and \ref{th-phipsi1}, we can find a complete Bernstein function 
$\psi_0$ such that $ \psi_0'(s)\sim \frac{a}{s\ell(1/s)}$,  
 $\psi_0 \sim \frac{a}{\theta(1/s)}$ at $0^+$  (for some $a>0$) and 
$$\phi_{\psi_0}(g)\simeq \frac{1}{(1+|g|)^D\ell (1+|g|^2)}.$$
This implies $\phi_{\psi_0}^{(n)}(e) \simeq \varphi ^{(n)}(e)$. 

Next, we claim that for any $\alpha\in (0,1)$, we can find a complete 
Bernstein function $\psi=\psi_\alpha$ such that 
$\psi\sim  b \psi_0^{1/\alpha}$, 
$\psi'\sim (b/\alpha) \psi_0'\psi_0^{-1+(1/\alpha)}$. 
If we set $\bar{\psi} =(\psi)^\alpha$ it then follows that  
$\bar{\psi}\sim \bar{a}\psi_0$  and
$(\bar{\psi})'\sim \bar{a}\psi_0'.$  If such a function exists, then we have:
\begin{itemize}
\item[(a)] By construction and Theorem \ref{th-PPpol}, for all $g\in G$ and $f\in L^2(G)$, we have
$$\sum_{x\in G}|f(xg)-f(x)|^2\le C \theta_2(|g|)^{1/\alpha} \mathcal E_{\phi_\psi}(f,f).$$
\item[(b)] By construction, $\phi_{\bar{\psi}}$ 
is the $\alpha$-subordinate of $\phi_\psi$.
\item[(c)] Since $(\bar{\psi})'\sim \bar{a}\psi_0'$, we have 
$$\phi_{\bar{\psi}}(g)\simeq \frac{1}{(1+|g|)^D\ell(1+|g|^2)}\simeq 
\phi_{\psi_0}(g) \simeq \varphi(g)$$
and, by \cite{PSCstab}, $\phi_{\bar{\psi}}^{(2n)}(e)\simeq 
\phi_{\psi_0}^{(2n)}(e)\simeq 
\varphi^{(2n)}(e).$
\end{itemize}
Using (a)-(b) and Theorem \ref{th-W}, we obtain that 
$\widetilde{\Phi}_{G,\theta_2}(n)\ge c\phi_{\bar{\psi}}^{(2Nn)}(e)$. 
Then (c) gives the desired inequality, 
$\widetilde{\Phi}_{G,\theta_2}(n)\ge c\varphi^{(2Nn)}(e)$. 

We are left with the task of constructing the appropriate complete 
Bernstein function $\psi=\psi_\alpha$, for each $\alpha\in (0,1)$. 
Since we want that $(\psi)^\alpha \simeq \psi_0$, the 
simple minded choice is to try $\psi=\psi_0^{1/\alpha}$. 
Unfortunately, this is not always a complete 
Bernstein function (because $1/\alpha>1$).  
However, in the present case, $\psi_1=\psi_0^{1/\alpha}$ has derivative
$\psi_1'=\alpha^{-1}\psi_0'\psi_0^{-1+(1/\alpha)}$. Hence
$$\psi_1'(s)\sim  \frac{a^{-1+(1/\alpha)}}{\alpha s \ell(1/s) \theta_2^{(1/\alpha)-1}(1/s)}.$$
Since $t\mapsto \ell(t)\theta_2^{(1/\alpha)-1}(t)$ is a continuous increasing slowly varying function, the desired complete Bernstein function $\psi$ is provided by Theorem \ref{th-cpsi}. 
\end{proof}

Together, Theorem \ref{th-BSC} and Theorem \ref{th-mainl} provide  
sharp results for a wide variety of regularly varying moment conditions
ranging through the entire index range $[0,2)$ in the context of groups 
of polynomial volume growth (see \cite{SCZ0} for sharp results regarding 
the special case $\alpha=2$). The results of 
\cite{BSClmrw} also provide sharp result in the case $\alpha\in (0,2)$
for groups of exponential volume growth such that 
$\Phi_G(n)\simeq \exp(-n^{1/3})$ 
(this covers all polycyclic groups 
with exponential volume growth). 

Results regarding slowly 
varying moment conditions for a variety of classes of groups with 
super-polymonial volume growth require different techniques and 
will be discussed elsewhere.

\bibliographystyle{amsplain}

\providecommand{\bysame}{\leavevmode\hbox to3em{\hrulefill}\thinspace}
\providecommand{\MR}{\relax\ifhmode\unskip\space\fi MR }
\providecommand{\MRhref}[2]{%
  \href{http://www.ams.org/mathscinet-getitem?mr=#1}{#2}
}
\providecommand{\href}[2]{#2}

\end{document}